\documentclass[11pt,a4paper]{amsart}
\usepackage[truedimen]{geometry}
\usepackage{graphicx}
\usepackage[utf8]{inputenc}
\usepackage{textcomp}
\usepackage{mathtools, amsmath, amssymb, mathrsfs, bm}
\usepackage{tikz, tikz-cd}
\usepackage{multirow}
\usepackage{url, hyperref, xcolor, comment, tikz-cd, dynkin-diagrams}
\hypersetup{
    colorlinks=true,
    citecolor=dgreen,
    linkcolor=dblue,
    urlcolor=black,
}

\definecolor{dgreen}{rgb}{0.15, 0.6, 0.15}
\definecolor{dblue}{rgb}{0.25, 0.45, 0.95}
\definecolor{dorange}{rgb}{0.95, 0.25, 0}

\usepackage{amsthm}
\theoremstyle{definition}
\newtheorem{counter}{counter}[section]
\newtheorem{definition}[counter]{Definition}
\newtheorem{theorem}[counter]{Theorem}
\newtheorem{proposition}[counter]{Proposition}
\newtheorem{lemma}[counter]{Lemma}

\newtheorem{remark}[counter]{Remark}
\newtheorem{claim}{Claim}
\newtheorem{corollary}[counter]{Corollary}
\newtheorem{example}[counter]{Example}

\newtheorem{convention}{Convention}

\newtheorem*{corollary*}{Corollary}

\newtheorem{problem}{Problem}

\newtheorem{maintheorem}{Main Theorem}

\usepackage{color}

\newcommand{\xto}{\xrightarrow}
\newcommand{\imm}{\looparrowright}
\newcommand{\into}{\hookrightarrow}

\renewcommand{\ge}{\geqslant}

\renewcommand{\epsilon}{\varepsilon}
\renewcommand{\phi}{\varphi}

\newcommand{\hy}{{\rm \mathchar`-}}
\renewcommand{\tilde}{\widetilde}

\newcommand{\Z}{\mathbb Z}
\newcommand{\R}{\mathbb R}
\newcommand{\C}{\mathbb C}
\newcommand{\HH}{\mathbb H}

\newcommand{\pt}{\mathrm{pt.}}

\newcommand{\SO}{\mathrm{SO}}

\newcommand{\SU}{\mathrm{SU}}
\newcommand{\U}{\mathrm{U}}

\newcommand{\Imm}{\mathrm{Imm}}
\newcommand{\Hom}{\mathrm{Hom}}
\newcommand{\Ker}{\operatorname{\mathrm{Ker}}}
\newcommand{\Mon}{\operatorname{\mathrm{Mon}}}

\renewcommand{\u}{\H{u}}

\makeatletter
\@namedef{subjclassname@2020}{\textup{2020} Mathematics Subject Classification}
\makeatother

\title[Immersions and Simple singularities]{Regular homotopy classes of links of simple singularities and immersions associated with their Dynkin diagrams}
\author{Masato Tanabe}
\email{tanabe.masato.i8@elms.hokudai.ac.jp}
\address{Department of Mathematics, Graduate School of Science, Hokkaido University, North 10, West 8, Kita-ku, Sapporo 060-0810, Japan}

\begin{document}

\maketitle

\begin{abstract}
Our aim is to determine the regular homotopy classes of immersions related to Arnol'd's simple singularities.
For every type of simple singularities, we determine the regular homotopy class of the inclusion map of the link into the 5-sphere.
We further show that the inclusion map is regularly homotopic to the immersion associated with the corresponding Dynkin diagram, which was constructed by Kinjo.
We prove these by computing the complete invariants of the immersions given by Wu and Saeki--Sz\u cs--Takase.
As an application, we also determine the Smale invariants of Kinjo's immersions.
\end{abstract}

\section{Introduction}\label{section:intro}

Given a complex analytic surface-germ in $\C^3$ with an isolated singularity, an oriented closed 3-manifold, the {\it link} of the singularity, appears within a small $5$-sphere centered at the singularity. The link is isotopic in the 5-sphere to the boundary of a complex analytic surface, called the {\it Milnor fiber} of the singularity. These subjects have been deeply studied in various contexts, such as singularity theory, low-dimensional topology, and contact/symplectic/complex geometry \cite{Mil68, CNP06, Sea06, Sea19}.

In this paper, we study links of isolated surface singularities in $\C^3$ from the viewpoint of immersions through two problems.
We first pose the following problem.

\begin{problem}\label{Q:incl}
{\it
Let $K$ be the link of an isolated surface singularity in $\C^3$.
Then which is the regular homotopy class of the inclusion map $f \colon K \into S^5_\epsilon$ as an immersion?
}
\end{problem}

Let us review the classification of immersions up to regular homotopy. 
In the late fifties, S.~Smale defined a complete invariant, now called the {\it Smale invariant}, for the classification of immersions of spheres into Euclidean spaces \cite{Sma59}. His result includes specific examples such as the classification of plane curves by Whitney--Graustein and the eversion of the 2-sphere in 3-space.
Shortly thereafter, M.~W.~Hirsch generalized Smale's result to immersions between arbitrary manifolds \cite{Hir59}. Thanks to their works, the classification problem of immersions up to regular homotopy was reduced to homotopy theory. These results are known as the Smale--Hirsch theory today.

Based on the theory, W.~T.~Wu studied immersions of oriented 3-manifolds into 5-space.
He gave a complete invariant for them up to regular homotopy. This invariant consists of two components.
The first one is an invariant on the 2-skeleton of the source 3-manifold --- this is named the {\it Wu invariant} in \cite{SST02}. 
The second one is an invariant on the 3-cells, which is an analogue of the Smale invariant --- let us call this the {\it Smale-type invariant}. 
If the normal bundle of a given immersion is trivial, then the Wu invariant is a second cohomology class of the 3-manifold which is a 2-torsion, and the Smale-type invariant is an integer. 
In order to determine the regular homotopy class of a given immersion, it is necessary to compute these two invariants.

However, it has been difficult to compute these invariants, since their geometric meanings were unclear from Wu's original work. The first attempt to resolve the difficulty was made by Saeki--Sz\u cs--Takase \cite{SST02}. They gave more geometric expressions of the two invariants for the case where the normal bundle is trivial. In particular, their work made the Smale-type invariant computable via certain (co)homological data. Their results have been applied to the study of immersions \cite{ST02}, generalized to the case where the normal bundle is non-trivial \cite{Juh05}, and also applied to the study of surface singularities in $\C^3$ \cite{GP24, PS24, PT23}.
In contrast, a framework to compute the Wu invariant has not yet been established. It seems to be due to the following reason: the Wu invariant is defined for a fixed parallelization (trivialization of the tangent bundle) of the 3-manifold. 
Furthermore, the Wu invariant takes all possible values by changing parallelizations (Appendix \ref{app:switch}).
Hence, it is crucial to choose a nice parallelization depending on situation.

Let us return to our first problem. We give an answer to Problem \ref{Q:incl} by computing the Wu and Smale-type invariants.

\begin{maintheorem}\label{main:determination-incl}
{\em
Under the setup of Problem \ref{Q:incl}, the Wu invariant $c_\tau(f)$ with respect to any almost contact parallelization $\tau$ and the Smale-type invariant $i(f)$ are
\[(c_\tau(f), i(f)) = \left(0, \frac{3}{2}(\sigma(F) - \alpha(K)) \right) \in H^2(K; \Z) \times \Z,\]
where $F$ is the Milnor fiber of the singularity, $\sigma(F)$ is the signature of $F$, and $\alpha(K)$ is an integer determined by the torsion part of $H_1(K; \Z)$ (see Definition \ref{dfn:alpha}).
}
\end{maintheorem}

Here is an outline of its proof.
The Smale-type invariant is computed by Saeki--Sz\u cs--Takase's formula. 
To compute the Wu invariant, we focus on the fact that the inclusion map $f \colon K \into S^5_\epsilon$ forms a contact embedding into the standard contact 5-sphere. 
In fact, this contact structure on $K$ can be trivialized as a 2-plane field by the result of Kasuya \cite{Kas16}. Hence we can parallelize $K$ along this trivialization. We call this kind of parallelization an {\it almost contact parallelization} in this paper (Definition \ref{dfn:al-ctct-paral}).
Then we can prove that the Wu invariant with respect to any almost contact parallelization vanishes for contact embeddings (Proposition \ref{Wuinv_ctct_emb}).
We note that Katanaga--N\'emethi--Sz\u cs studied a problem corresponding to Problem \ref{Q:incl} for hypersurface singularities in $\C^4$ of type $A$ \cite{KNS14}.

\vspace{6pt}

Our second problem focuses on {\it simple singularities}, the most fundamental class in the classification of holomorphic function-germs. V.~I.~Arnol'd classified them into types of $A$, $D$, and $E$ \cite{AGV85}. It is known that their links admit several special structures, such as quotient spaces under actions of finite subgroups of $\SU(2)$ on $S^3$ \cite{Kle56} and plumbed 3-manifolds along Dynkin diagrams.
These manifolds and structures also have been deeply studied.
Now, we recall two works which motivated our problem.

For any holomorphic map-germ $\Phi \colon (\C^2, 0) \to (\C^3, 0)$ which is singular only at the origin, A.~N\'emethi and G.~Pint\'er \cite{NP15} obtained an immersion $\Phi|_\mathfrak{S} \colon S^3 \imm S^5_\epsilon$ by taking the preimage $\mathfrak{S} = \Phi^{-1}(S^5_\epsilon) \cong S^3$.
They discovered a formula of the Smale invariant of $\Phi|_\mathfrak{S}$ by counting complex singularities of a holomorphic perturbation of $\Phi$.
Moreover, as examples of map-germs $\Phi$, they considered parametrizations of simple singularities.
Then they determined the regular homotopy classes of $\Phi|_\mathfrak{S}$ for all $A$-$D$-$E$ cases.
Note that in these cases, each immersion is factorized into $\Phi|_\mathfrak{S} = f \circ p$, where $f \colon K \into S^5_\epsilon$ is the inclusion map of the link and $p \colon S^3 \to K$ is the universal covering map.

S.~Kinjo \cite{Kin15} associated an immersion to each of the Dynkin diagrams of type $A$-$D$-$E$.
The construction is as follows: let $M(G)$ denote the plumbed 3-manifold along the diagram $G$, which is orientation-preservingly diffeomorphic to the link of the corresponding singularity. 
Construct an immersion $\bar{g}_G = g_G \circ r \colon M(G) \imm \R^4$, where $r \colon M(G) \to M(G)$ is an orientation-reversing diffeomorphism. Pull-back $\bar{g}_G$ by the universal covering map $p \colon S^3 \to M(G)$.
Then she determined the regular homotopy classes of $\bar{g}_G \circ p$ for the $A$-$D$ cases by computing their Smale invariants. She employed Ekholm--Takase's formula \cite{ET11}, constructing certain $C^\infty$ maps so-called singular Seifert surfaces of the immersions and counting their real singularities.
However, the method did not apply to the $E$ cases, leaving their Smale invariants unknown.

The works of N\'emethi--Pint\'er and Kinjo had a connection despite the difference of their natures. They pointed out it as follows: let $j \colon \R^4 \into \R^5$ denote the standard inclusion. Then, for each of the $A$-$D$ cases, the Smale invariants of $f \circ p$ and $j \circ \bar{g}_G \circ p$ coincide with opposite sign.
This implies that $f \circ p$ and $j \circ g_G \circ p$ are regularly homotopic (under removing one point from $S^5_\epsilon$). 
Then the following problem naturally arises.

\begin{problem}\label{Q:incl-and-pushedKinjo}
{\it
Consider any of the singularities of type $E$.
\begin{itemize}
\item[(a)] Which is the regular homotopy class of Kinjo's immersion $\bar{g}_G \circ p \colon S^3 \imm \R^4$?
\item[(b)] Are the two immersions $f \circ p$ and $j \circ g_G \circ p \colon S^3 \imm \R^5$ regularly homotopic?
\end{itemize}
}
\end{problem}

We solve these by showing that {\it two immersions $f$ and $j \circ g_G$ are actually regularly homotopic}. More precisely, we show the following assertion.

\begin{maintheorem}\label{main:determination-Kinjo}
{\em
For each type of simple singularities and the corresponding Dynkin diagram $G$, 
the Wu invariant $c_\tau(j \circ g_G)$ with respect to any almost contact parallelization $\tau$ and the Smale-type invariant $i(j \circ g_G)$ are
\[(c_\tau(j \circ g_G), i(j \circ g_G)) = \left(0, \frac{3}{2}(-\# V(G) - \alpha(M(G))) \right) \in H^2(M(G); \Z) \times \Z,\]
where $\# V(G)$ is the number of vertices of $G$.
}
\end{maintheorem}

As in the proof of Main Theorem \ref{main:determination-incl}, we compute the Smale-type invariant by Saeki--Sz\u cs--Takase's formula. 
To compute the Wu invariant, we consider an almost contact structure on $M(G)$ and an almost contact parallelization along this (\S \ref{subsection:reconstruction-identification}). However, the reason for the vanishing of the Wu invariant differs from the previous case. We regularly homotope the immersion $g_G$ so that the image of the almost contact structure forms complex tangency for the standard complex structure on $\R^4 = \C^2$. 
Then we prove that if we push-forward this immersion into $\R^5$, then the Wu invariant with respect to the almost contact parallelization vanishes (Proposition \ref{Wuinv_hol_imm}).

In fact, we can identify two manifolds $K$ and $M(G)$ as almost contact manifolds (\S \ref{subsubsection:identification}).
Then Main Theorems \ref{main:determination-incl} and \ref{main:determination-Kinjo} implies the following, by the coincidence of their complete invariants: for each of $A$-$D$-$E$ cases, the inclusion map $f \colon K \into S^5_\epsilon$ and the immersion $j \circ g_G \colon M(G) \imm \R^5$ are regularly homotopic (Corollary \ref{cor:reg-htpicity}).
It is immediate that two immersions $f \circ p$ and $j \circ g_G \circ p$ are regularly homotopic. In particular, we obtain the affirmative answer to Problem \ref{Q:incl-and-pushedKinjo}(b).

As an application of this regular homotopicity, we also determine the regular homotopy classes of Kinjo's immersions $\bar{g}_G \circ p$ for all $A$-$D$-$E$ cases, by computing the Smale invariants. 
This result recovers that of Kinjo for the $A$-$D$ cases, and also gives the answer to Problem \ref{Q:incl-and-pushedKinjo}(a).

\begin{corollary*}[$=$ Theorem \ref{thm:determination-Kinjo-Smale}]\label{main:determination-Kinjo-Smale}
{\em
For each type of simple singularities and the corresponding Dynkin diagram $G$, the Smale invariant $\Omega(\bar{g}_G \circ p) \in \pi_3(\SO(4)) \cong \Z \oplus \Z$ (see \S \ref{subsection:S3R4} for the choice of generators) is
\begin{align*}
\Omega(\bar{g}_G \circ p) &= 
\begin{cases}
(n^2 - 1, 0) & \text{ for } A_{n-1} \ (n \ge 2); \\
(4n^2 + 12n - 1, 0) & \text{ for } D_{n+2} \ (n \ge 2); \\
(167, 0) & \text{ for } E_6; \\
(383, 0) & \text{ for } E_7; \\
(1079, 0) & \text{ for } E_8.
\end{cases}
\end{align*}
}
\end{corollary*}

To prove this, we additionally employs N\'emethi--Pint\'er's result and the computation of the normal mapping degree of $\bar{g}_G \circ p$. This argument is quite different from that of Kinjo (Remark \ref{rmk:argument}).
Note that it was expected by Pint\'er to find any direct relationship between N\'emethi--Pint\'er's and Kinjo's immersions in his thesis \cite{Pin18}.
We realized the expectation and provided new aspects for the immersions, in the above sense.

We end the introduction, summarizing our study in Figure \ref{fig:overview}.

\begin{figure}[ht]
\begin{tikzpicture}[boxnode/.style={shape=rectangle, draw=black, fill=white, text centered, minimum height=10mm}]
   \node[boxnode, text width=50mm] (a) at (0,3) {$f \colon K \into S^5_\epsilon$\\Main Theorem \ref{main:determination-incl} for $A$-$D$-$E$}; 
   \node[boxnode, text width=50mm] (b) at (6.2,3) {$\Phi|_{\mathfrak{S}} = f \circ p \colon S^3 \imm S^5_\epsilon$\\\cite{NP15} for $A$-$D$-$E$};
   \node[boxnode, text width=50mm] (c) at (0,0) {$j \circ g_G \colon M(G) \imm \R^5$\\Main Theorem \ref{main:determination-Kinjo} for $A$-$D$-$E$}; 
   \node[boxnode, text width=50mm] (d) at (6.2,0) {$j \circ g_G \circ p \colon S^3 \imm \R^5$\\\cite{Kin15} for $A$-$D$, \\ Corollary \ref{cor:reg-htpicity} for $A$-$D$-$E$}; 
   \node[boxnode, text width=35mm] (e) at (-3,-2.5) {$g_G \colon M(G) \imm \R^4$}; 
   \node[boxnode, text width=50mm] (f) at (3.5,-2.5) {$g_G \circ p \colon S^3 \imm \R^4$\\\cite{Kin15} for $A$-$D$, \\ Theorem \ref{thm:determination-Kinjo-Smale} for $A$-$D$-$E$}; 
   \draw[->] (a) -- node[yshift=5pt] {$\circ p$} (b);
   \draw[->] (c) -- node[yshift=5pt] {$\circ p$} (d);
   \draw[->] (e) -- node[yshift=5pt] {$\circ p$} (f);
   \draw[->] (e) -- node[xshift=-10pt] {$j \circ$} (c);
   \draw[->] (f) -- node[xshift=-10pt] {$j \circ$} (d);
   \draw[<->, dashed] (a) -- node[align=center,xshift=-70pt] {{\it regularly homotopic} \\ Corollary \ref{cor:reg-htpicity} for $A$-$D$-$E$} (c);
   \draw[<->, dashed] (b) -- node[align=center,xshift=70pt] {{\it regularly homotopic} \\ \cite{Kin15, NP15} for $A$-$D$, \\ Corollary \ref{cor:reg-htpicity} for $A$-$D$-$E$} (d);
\end{tikzpicture}
\caption{Overview of our study}
\label{fig:overview}
\end{figure}

\subsection*{Organization}

In \S \ref{section:surf-sing}, we recall basic notions on simple singularities and introduce immersions which are our main subjects.
\S \ref{section:imm-generality} and \S \ref{section:al_ctct} are devoted to recall immersions of oriented 3-manifolds and the notion of almost contact structures.
In \S \ref{section:Wuinv}, we show general properties of the Wu invariant.
In \S \ref{section:main}, we show Main Theorems \ref{main:determination-incl} and \ref{main:determination-Kinjo} and give their applications.
In Appendix \ref{app:switch}, we discuss further properties of the Wu invariant.

In this paper, by {\it manifold} and {\it map}, we always mean that of class $C^\infty$.

\subsection*{Acknowledgement}

The author would like to thank Naohiko Kasuya, Jiro Adachi, Sakumi Sugawara, Gerg\H{o} Pint\'er for helpful communications.
Naohiko Kasuya gave me helpful advice, discussions, and encouragements throughout this study.
Jiro Adachi and Sakumi Sugawara suggested me to study the topics in Appendix \ref{app:switch}.
Gerg\H{o} Pint\'er kindly informed me numerous literature, noticed the fact stated in Remark \ref{rmk:intrinsic}, and suggested and helped me to show Theorem \ref{thm:determination-Kinjo-Smale}.

The author also wishes to thank the members of Saturday Topology Seminar, Toru Ohmoto, and an anonymous referee for giving invaluable comments on the early and first versions of this paper.

This work was supported by JST SPRING, Grant Number JPMJSP2119.

\section{Simple singularities and their immersions}\label{section:surf-sing}

This section is devoted to review simple singularities and introduce immersions related to them, which are our main subjects.
See \cite{Mil68, Sea06, Kas15, Sea19} for details.

\subsection{Basic notions}

Let $h \colon (\C^3, 0) \to (\C, 0)$ be a holomorphic function-germ with an isolated singularity at the origin.
The germ itself or its zero set-germ $(h^{-1}(0), 0) \subset (\C^3, 0)$ are called an {\it isolated surface singularity in $\C^3$}.

\begin{definition}[link and Milnor fiber]
For a sufficiently small number $0 < \epsilon \ll 1$, the intersection $K \coloneqq h^{-1}(0) \cap S^5_\epsilon$ forms a closed oriented 3-manifold embedded into $S^5_\epsilon$. We call $K$ the {\it link} of the singularity.
In addition, for a sufficiently small number $0 < \delta \ll \epsilon$, the intersection $F \coloneqq h^{-1}(\delta) \cap B^6_\epsilon$ forms a compact oriented 4-manifold with boundary, properly embedded into $B^6_\epsilon$. We call $F$ the {\it Milnor fiber} of the singularity.
\end{definition}

Notice that the link $K$ is isotopic to the boundary $\partial F$ of the Milnor fiber in $S^5_\epsilon$.
We often identify them throughout this paper.

\subsection{Simple singularities}

For singularities of function-germs $(\C^3, 0) \to (\C, 0)$, Arnol'd introduced the notion of modality and classified the singularities up to modality 0, 1, and 2 \cite{AGV85}.
Singularities of modality 0, which appear on the most earlier hierarchy in the classification, are called {\it simple singularities}, {\it Kleinian singularities}, {\it du Val singularities}, or {\it rational double points}.
All types of simple singularities are listed as follows.

\begin{itemize}
\item $A_{n-1}$ ($n \ge 2$): $x^2+y^2+z^n$
\item $D_{n+2}$ ($n \ge 2$): $x^2+y^2z+z^{n+1}$
\item $E_6$: $x^2+y^3+z^4$
\item $E_7$: $x^2+y^3+yz^3$
\item $E_8$: $x^2+y^3+z^5$  
\end{itemize}

As is well-known, their links and Milnor fibers have geometric structures, and can be studied from various viewpoints. Let us review specific two.

\subsubsection{Quotient structures}\label{subsubsection:quotient}

Klein showed that for each type of simple singularities, their zero set-germ in $(\C^3, 0)$ is isomorphic to the quotient of $(\C^2, 0)$ by a finite subgroup of $\SU(2)$  \cite{Kle56} (see also \cite{Mil75}).
The corresponding subgroups are listed as follows.

\begin{itemize}
\item $A_{n-1}$ ($n \ge 2$): the cyclic group $C_n$ of order $n$
\item $D_{n+2}$ ($n \ge 2$): the binary dihedral group $\mathrm{Dic}_n$ of order $4n$
\item $E_6$: the binary tetrahedral group $2T$, whose order is 24
\item $E_7$: the binary octahedral group $2O$, whose order is 48
\item $E_8$: the binary icosahedral group $2I$, whose order is 120
\end{itemize}
Let $\Gamma \subset \SU(2)$ be one of these groups.
The left action of $\Gamma$ on $\C^2$ can be restricted to a free action on the 3-sphere $S^3 \subset \C^2$.
Then the link $K$ is diffeomorphic to the quotient space $\Gamma \backslash S^3$, i.e., the link $K$ admits the universal covering map $p \colon S^3 \to K$ whose covering degree is the order of the group $\Gamma$.
Note that by the invariant polynomial theory, holomorphic map-germs $\Phi \colon (\C^2, 0) \to (\C^3, 0)$ parametrizing simple singularities are known,
e.g.,
\[(s, t) \mapsto \left( \frac{1}{2}(s^n-t^n), \frac{\sqrt{-1}}{2}(s^n+t^n), st \right), \quad (s, t) \mapsto \left( \frac{st}{2}(s^{2n}-t^{2n}), \frac{\sqrt{-1}}{2}(s^{2n}+t^{2n}), s^2t^2 \right)\]
for the singularities of types $A_{n-1}$ and $D_{n+2}$, respectively. 
These will appear in \S \ref{subsubsection:imm-NP}.

\subsubsection{Plumbed manifold structures}

For every type of simple singularities, the link admits the plumbed 3-manifold structure corresponding to the weighted dual graph of the minimal resolution of the singularity.
The underlying graphs are called {\it Dynkin diagrams}, which are listed below.

\begin{figure}[hb]
\label{fig:dynkin-diagram}

\begin{center}
$A_{n-1}: \ $ \dynkin[radius=100mm, edge length=6mm]{A}{} 
$\quad D_{n+2}: \ $ \dynkin[radius=100mm, edge length=6mm]{D}{}
\end{center}

\begin{center}
$E_6: \ $ \dynkin[radius=100mm, edge length=6mm]{E}{6} 
$\quad E_7: \ $ \dynkin[radius=100mm, edge length=6mm]{E}{7}
$\quad E_8: \ $ \dynkin[radius=100mm, edge length=6mm]{E}{8}
\end{center}

\caption{Dynkin diagrams of types $A$-$D$-$E$: weights of vertices and edges are all $-2$ and $1$, respectively.}
\end{figure}

\noindent{Let $G$ be one of Dynkin diagrams of types $A$-$D$-$E$, and let $M(G)$ (resp.~$X(G)$) denote the plumbed 3- (resp.~4-) manifold corresponding to $G$.} 
Notice that each component of the plumbing is the unit cotangent bundle $UT^*S^2$ (resp.~the disk cotangent bundle $DT^*S^2$).
For the type of simple singularities corresponding to $G$, its minimal resolution is diffeomorphic to $X(G)$. Hence, the link $K$ is diffeomorphic to $M(G)$.

As a specific property of simple singularities, the Milnor fiber $F$ is also diffeomorphic to $X(G)$.
However, the complex structures on the minimal resolution and the Milnor fiber are quite different in the following sense.
Each manifold admits a spine which is expressed as the transverse union of 2-spheres.
Spheres in the spine can be taken as complex submanifolds in the minimal resolution (the union of exceptional divisors), 
while totally real submanifolds in the Milnor fiber (the section by real $3$-space).
This difference will be crucial in \S \ref{subsection:reconstruction-identification}, \ref{subsection:Wuinv-cpx-tan}.

\subsection{Immersions of links of simple singularities}

We list all immersions which are studied in this paper as follows.
See also Figure \ref{fig:overview} in \S \ref{section:intro}.

\subsubsection{Inclusion maps}\label{subsubsection:imm-incl}

Let us consider an isolated surface singularity in $\C^3$, its Milnor fiber $F$, and its link $K$.
By definition, we have inclusion maps
\[\hat{f} \colon F \into B^6_\epsilon \quad \text{and} \quad f \coloneqq \hat{f}|_K \colon K \into S^5_\epsilon.\]
Here we identified $\partial F$ with $K$, since they are isotopic in $S^5_\epsilon$.
Removing one point of $S^5_\epsilon$ missed by $f$, we obtain embeddings which are denoted by
\[\hat{f} \colon F \into \R^6_+ \quad \text{and} \quad f \coloneqq \hat{f}|_K \colon K \into \R^5\] 
as well.

\begin{remark}\label{rmk:remove-one-point}
The classification of 3-manifolds into 5-space up to regular homotopy is completely the same as that into the 5-sphere. Hence, the removement of one point is not necessary, and just for the convenience of our argument. 
\end{remark}

\subsubsection{Immersions defined by holomorphic map-germs}\label{subsubsection:imm-NP}

The following construction is due to N\'emethi--Pint\'er \cite{NP15}.
Let $\Phi \colon (\C^2, 0) \to (\C^3, 0)$ be a holomorphic map-germ which is singular only at the origin.
Then, for a sufficiently small number $\epsilon > 0$, the preimage $\mathfrak{S} \coloneqq \Phi^{-1}(S^5_\epsilon)$ is canonically diffeomorphic to the 3-sphere $S^3$.
Hence, the restriction of $\Phi$ to $\mathfrak{S}$ forms an immersion
\[\Phi|_\mathfrak{S} \colon S^3 \imm S^5_\epsilon\]
(its regular homotopy class is independent of the choice of $\epsilon$ and the orientation-preserving diffeomorphism $\mathfrak{S} \cong S^3$).
If the map $\Phi$ is the parametrization of a given type of simple singularities and $K$ is its link, then the induced immersion is factorized into
\[\Phi|_\mathfrak{S} = f \circ p \colon S^3 \to K \into S^5_\epsilon,\]
where $f \colon K \into S^5_\epsilon$ is the inclusion map of the link and $p \colon S^3 \to K$ is the universal covering map.

\subsubsection{Immersions associated with Dynkin diagrams of types $A$-$D$-$E$}\label{subsubsection:imm-Kinjo}

The following construction is due to Kinjo \cite{Kin15}.
We consider an immersion $S^2 \imm \R^4$ with only one positively transverse double point. Since the total space of its disk normal bundle is the disk cotangent bundle $DT^*S^2 = X(A_1)$, we obtain immersions
\[\hat{g}_{A_1} \colon X(A_1) \imm \R^4 \quad \text{and} \quad g_{A_1} \coloneqq \hat{g}_{A_1}|_{M(A_1)} \colon M(A_1) \imm \R^4.\]
Furthermore, for each Dynkin diagram $G$ of types $A$-$D$-$E$, we construct immersions
\[\hat{g}_G \colon X(G) \imm \R^4 \quad \text{and} \quad g_G \coloneqq \hat{g}_G|_{M(G)} \colon M(G) \imm \R^4\]
by plumbing the copies of $\hat{g}_{A_1}$ according to $G$.
We also pull-back the immersions $g_G$ by the universal covering map $p \colon S^3 \to M(G)$. Then we obtain immersions $g_G \circ p \colon S^3 \imm \R^4$.

We will reconstruct (regularly homotope) immersions $g_G$ as nicer ones, so that they fit to almost contact and complex structures (see \S \ref{subsection:reconstruction-identification}).

\begin{remark}\label{rmk:orientation-of-source}
To be precise, the immersion Kinjo constructed is $\bar{g}_G \coloneqq g_G \circ r \colon M(G) \imm \R^4$, where $r \colon M(G) \to M(G)$ is an orientation-reversing diffeomorphism.
For example, the 3-manifold Kinjo immersed is orientation-preservingly diffeomorphic to $L(n, 1)$.
However, the link $K$ of type $A_{n-1}$ is endowed with the orientation as $L(n, n-1)$, which is induced from the standard complex structure on the Milnor fiber $F \subset \C^3$.
\end{remark}

\section{Immersions of oriented 3-manifolds}\label{section:imm-generality}

In this section, we review the works on the classification of immersions of 3-manifolds into 4- and 5-spaces up to regular homotopy, and their formulae obtained in accordance with the methodology of singular Seifert surfaces.
We omit definitions of some notions which will not be used in our results.
See the original papers and \cite{Pin18} for details.

Given manifolds $M$ and $N$, let $\Imm[M, N]$ denote the set of all regular homotopy classes of immersions of $M$ into $N$.
Let $V_{n,k}$ denote the space of all orthonormal $k$-frames in $\R^n$, which is called the Stiefel manifold.

\subsection{Immersions of the 3-sphere into 4-space}\label{subsection:S3R4}

We begin with the lowest-dimensional case.
Smale showed that immersions of the $k$-sphere into $n$-space ($k < n$) are classified by the $k$-th homotopy group of the Stiefel manifold $V_{n,k}$ up to regular homotopy \cite{Sma59}. To be precise, he constructed a group isomorphism
\[\Omega \colon \Imm[S^k, \R^n] \xto{\cong} \pi_k(V_{n,k}),\]
where the group structure on $\Imm[S^k, \R^n]$ is given by the connected sum.
Given an immersion, its value under this map is called the {\it Smale invariant} of the immersion.
In particular, we have that
\[\Imm[S^3, \R^4] \cong \pi_3(V_{4,3}) \cong \pi_3(\SO(4)) \cong \Z \oplus \Z.\]
We fix the last isomorphism as follows, referring to \cite{Ste51, ET11, Kin15, NP15}:
we identify $\R^4$ with the set $\HH$ of all quaternions by identifying standard basis $e_1, e_2, e_3, e_4$ with $1, i, j, k$, respectively. 
We also identify $S^3$ with the set of all unit quaternions as the restriction of the above.
For $x, y \in \HH$, let $x \cdot y$ denote their quaternionic product.
Now we consider the principal $\SO(3)$-bundle $\pi \colon \SO(4) \to S^3, \ \pi(R) \coloneqq R(e_1)$.
Then the map
\[\sigma \colon S^3 \to \SO(4), \ \sigma(x)(y) \coloneqq x \cdot y\]
forms a section of $\pi$. Hence, $\pi \colon \SO(4) \to S^3$ is trivial, i.e., $\SO(4) \cong S^3 \times \SO(3)$.
Furthermore, the map
\[\rho \colon S^3 \to \SO(4), \ \rho(x)(y) \coloneqq x \cdot y \cdot x^{-1}\]
forms a universal covering map over the fiber $\SO(3) \subset \SO(4)$.
Therefore, we have that
\[\pi_3(\SO(4)) \cong \pi_3(S^3 \times \SO(3)) \cong \pi_3(S^3) \oplus \pi_3(\SO(3)) = \Z [\sigma] \oplus \Z [\rho] \cong \Z \oplus \Z,\] 
where the last isomorphism is the standard identification.

For an immersion $g \colon S^3 \imm \R^4$, a {\it singular Seifert surface} of $g$ is a generic map $\hat{g} \colon X^4 \to \R^4$ from a compact oriented 4-manifold $X^4$ bounded by $S^3$ satisfying that $\hat{g}|_{S^3} = g$ and has no singularities near the boundary.
Ekholm--Takase presented the following formula to compute the Smale invariant (two integers), based on Hughes' work.

\begin{theorem}[\cite{Hug92, ET11}]\label{thm:ET11}
{\it 
Let $g \colon S^3 \imm \R^4$ be an immersion and $\hat{g} \colon X^4 \to \R^4$ a singular Seifert surface of $g$.
Then the Smale invariant of $g$ is given by
\[\Omega(g) = \left( D(g) - 1,  \frac{3 \sigma(X^4) + \# \Sigma^2(\hat{g}) - 2(D(g) - 1)}{4} \right),\]
where $D(g)$ is the normal mapping degree of $g$ and $\# \Sigma^2(\hat{g})$ is the algebraic number of $\Sigma^2$-singularities of $\hat{g}$ (see \cite[\S 2.3]{ET11}).
}
\end{theorem}

Kinjo applied Theorem \ref{thm:ET11} to the following computation.

\begin{theorem}[\cite{Kin15}]\label{thm:Kin15}
{\em
Consider the Dynkin diagram $G$ of type $A$ or $D$, Kinjo's immersion $\bar{g}_G \colon M(G) \imm \R^4$, and the universal covering map $p \colon S^3 \to M(G)$. 
Then the Smale invariant of $\bar{g}_G \circ p \colon S^3 \imm \R^4$ is
\[
\Omega(\bar{g}_G \circ p) = 
\begin{cases}
(n^2 - 1, 0) & \text{ for } A_{n-1} \ (n \ge 2); \\
(4n^2 + 12n - 1, 0) & \text{ for } D_{n+2} \ (n \ge 2).
\end{cases} 
\]
}
\end{theorem}

We note on changing of the Smale invariant of $S^3 \imm \R^4$ by reversing the orientation.

\begin{proposition}[{cf.~\cite[Lemma 2.5]{Hug92}}]
{\em
Let $g \colon S^3 \imm \R^4$ be an immersion and $r \colon S^3 \to S^3$ an orientation-reversing diffeomorphism.
Then it holds that
\[\Omega(g \circ r) = -\Omega(g) + (-2, 1).\]
}
\end{proposition}

Then we can deduce the following from Theorem \ref{thm:Kin15}.

\begin{corollary}\label{thm:Kin15-rev}
{\em
Under the same setup of Theorem \ref{thm:Kin15}, the Smale invariant of $g_G \circ p$ is
\[
\Omega(g_G \circ p) = 
\begin{cases}
(-n^2 - 1, 1) & \text{ for } A_{n-1} \ (n \ge 2); \\
(-4n^2 - 12n - 1, 1) & \text{ for } D_{n+2} \ (n \ge 2).
\end{cases}
\]
}
\end{corollary}

\subsection{Immersions of the 3-sphere into 5-space}\label{subsection:S3R5}

According to Smale, we also have that
\[\Imm[S^3, \R^5] \cong \pi_3(V_{5,3}) \cong \pi_3(\SO(5)) \cong \Z.\]
We fix the last isomorphism by choosing the generator of $\pi_3(\SO(5))$ to be
\[\iota \circ \sigma \colon S^3 \to \SO(4) \into \SO(5),\]
where $\sigma \colon S^3 \to \SO(4)$ is the map defined in \S \ref{subsection:S3R4} and $\iota \colon \SO(4) \into \SO(5)$ is the standard inclusion (see also Remark \ref{rmk:sign1} below).

For an immersion $f \colon S^3 \imm \R^5$, a {\it singular Seifert surface} of $f$ is a generic map $\hat{f} \colon X^4 \to \R^6_+$ from a compact oriented 4-manifold $X^4$ bounded by $S^3$ satisfying that $\hat{f}^{-1}(\R^5) = S^3$, $\hat{f}|_{S^3} = f$, and $\hat{f}$ has no singularities near the boundary.
Ekholm--Sz\u cs presented the following formula to compute the Smale invariant (one integer), based on Hughes--Melvin's work.

\begin{theorem}[\cite{HM85, ES03}, see also \cite{SST02}]\label{thm:ES03}
{\it 
Let $f \colon S^3 \imm \R^5$ be an immersion and $\hat{f} \colon X^4 \to \R^6_+$ a singular Seifert surface of $f$.
Then the Smale invariant of $f$ is given by
\[\Omega(f) = \frac{3}{2}\sigma(X^4) + \frac{1}{2}(3t(\hat{f}) - 3l(\hat{f}) + L(f)),\]
where $t(\hat{f})$, $l(\hat{f})$, and $L(f)$ are integers obtained from the singularity of $\hat{f}$ and the double locus of $f$ (see \cite[\S 2.2 and 2.3]{ES03}, and also \cite{PS24, PT23} for $L(f)$).
}
\end{theorem}

N\'emethi--Pint\'er studied the Smale invariant from in another way, and showed the following. Recall Remark \ref{rmk:remove-one-point}.

\begin{theorem}[\cite{NP15}]\label{thm:cpx-Smale}
{\it 
Let $\Phi \colon (\C^2, 0) \to (\C^3, 0)$ be a holomorphic map-germ which is singular only at the origin.
Consider the immersion $\Phi|_\mathfrak{S} \colon S^3 \imm S^5$ introduced in \S \ref{subsubsection:imm-NP}.
Then the Smale invariant of $\Phi|_\mathfrak{S}$ is given by
\[\Omega(\Phi|_\mathfrak{S}) = -C(\Phi),\]
where $C(\Phi)$ is the number of complex Whitney umbrella singularities appearing in a holomorphic stable perturbation of $\Phi$ (see \cite[\S 2.2]{NP15}).
}
\end{theorem}


\begin{remark}\label{rmk:sign1}
They also pointed out the ambiguity of the choice of the generator of $\pi_3(\SO(5))$ (the sign of the Smale invariant), and fixed it.
They showed that the sign in Theorems \ref{thm:ES03} and \ref{thm:cpx-Smale} agree each other (see \cite[Theorem 9.1.6]{NP15}, and also \cite[Appendix A.3]{PS24}).
\end{remark}

They applied Theorem \ref{thm:NP15} to the following computation.

\begin{theorem}[\cite{NP15}]\label{thm:NP15}
{\em
Consider any type of simple singularities, the parametrization $\Phi \colon (\C^2, 0) \to (\C^3, 0)$ of its zero set-germ, and the immersion $\Phi|_\mathfrak{S} = f \circ p \colon S^3 \imm S^5$.
Then the Smale invariant of $\Phi|_\mathfrak{S}$ is
\[\Omega(\Phi|_\mathfrak{S}) = 
\begin{cases}
-(n^2 - 1) & \text{ for } A_{n-1} \ (n \ge 2); \\
-(4n^2 + 12n - 1) & \text{ for } D_{n+2} \ (n \ge 2); \\
-167 & \text{ for } E_6; \\
-383 & \text{ for } E_7; \\
-1079 & \text{ for } E_8.
\end{cases}
\]
}
\end{theorem}

We also note on the relationship between the Smale invariants of $\Imm[S^3, \R^4]$ and $\Imm[S^3, \R^5]$.
By the standard inclusion $j \colon \R^4 \into \R^5$, we consider the opretation 
\[j_* \colon \Imm[S^3, \R^4] \to \Imm[S^3, \R^5], \quad g \mapsto j \circ g,\]
which forms a group homomorphism. Then the following is known.

\begin{proposition}[{\cite[\S 22.7, 23.6]{Ste51}, see also \cite[Lemma 2.4]{Hug92}}]\label{thm:relationship-push}
{\em 
Regard the map $j_*$ as $\Z \oplus \Z \to \Z$ via our choices of generators. 
Then it holds that
\[j_*(a, b) = a + 2b.\]
}
\end{proposition}

Consequently, we have the following as a corollary of Corollary \ref{thm:Kin15}, Theorem \ref{thm:NP15}, and Proposition \ref{thm:relationship-push}. It was pointed out by \cite{Kin15, NP15}.

\begin{corollary}\label{cor:reg-htpicity-AD}
{\em
Consider one type of simple singularities and the corresponding Dynkin diagram $G$.
If the singularity is of type $A$ or $D$, then two immersions $\Phi|_\mathfrak{S} = f \circ p \colon S^3 \imm \R^5$ and $j \circ g_G \circ p \colon S^3 \imm \R^5$ are regularly homotopic.
}
\end{corollary}

We will refine and generalize this fact to the $E$ case in \S \ref{section:application}.

\subsection{Immersions of oriented 3-manifolds into 5-space}\label{subsection:ImmM3R5}

Let $M^3$ denote an arbitrary oriented 3-manifold (in \S \ref{subsection:Smaletype}, we will also assume that $M^3$ is closed and connected).
As is well-known, $M^3$ is parallelizable, i.e., the tangent bundle $TM^3$ can be trivialized.
{\it We fix a parallelization $\tau \colon TM^3 \xto{\cong} M^3 \times \R^3$ of $M^3$.}
Then for an immersion $f \colon M^3 \imm \R^5$, its differential $df \colon TM^3 \to T\R^5$ is represented as
\[A_{\tau, f} \colon M^3 \to (\R^5)^3, \quad x \mapsto \begin{bmatrix} df_x(e_1) & df_x(e_2) & df_x(e_3) \end{bmatrix},\]
where $e_i$ is the tangent vector on $M^3$ which corresponds to the $i$-th standard vector via $\tau$.
By orthonormalizing $A_{\tau, f}$, we have the map
\[\phi_{\tau, f} \colon M^3 \to V_{5,3} \cong \SO(5)/\SO(2).\]

We employ the Smale--Hirsch theory, which asserts the following.

\begin{theorem}[\cite{Sma59, Hir59}]
{\it 
Let $M$ and $N$ be manifolds such that $\dim M < \dim N$.
Let $\Imm(M, N)$ denote the space of all immersions of $M$ into $N$, and $\Mon(TM, TN)$ the space of all fiberwise injective homomorphisms between the tangent bundles $TM$ and $TN$. 
We endow $\Imm(M, N)$ with the $C^\infty$ topology, and $\Mon(TM, TN)$ with the compact open topology.
Then the natural inclusion
\[\Imm(M, N) \into \Mon(TM, TN), \quad f \mapsto df\]
is weak homotopy equivalence.
}
\end{theorem}

Then we have that the correspondence
\[\Imm[M^3, \R^5] \to [M^3, V_{5,3}], \quad [f] \mapsto [\phi_{\tau, f}]\]
is a bijection.
By homotopy-theoretical arguments based on this bijection, Wu concluded the following.
For a class $\chi \in H^2(M^3; \Z)$, let $\Imm[M^3, \R^5]_\chi$ denote the set of all regular homotopy classes of immersions of $M^3$ into $\R^5$ with normal Euler class $\chi$.

\begin{theorem}[{\cite[Theorem 2]{Wu64}, see also \cite{Li82, SST02, Juh05}}]\label{Wu}
{\em
For any immersion $f \colon M^3 \imm \R^5$, its normal Euler class is of the form $2C \in H^2(M^3; \Z)$, where $C \in H^2(M^3; \Z)$.
Furthermore, for any $\chi \in H^2(M^3; \Z)$, there is a bijection 
\[\Imm[M^3, \R^5]_\chi \to \Gamma_2(\chi) \times H^3(M^3; \Z)/(2\chi \smile H^1(M^3; \Z)),\]
where
\[\Gamma_2(\chi) \coloneqq \{ C \in H^2(M^3; \Z) \mid 2C = \chi \}.\]
}
\end{theorem}

We explain two components appearing in the right-hand side of this bijection.

\subsubsection{Wu invariant}

The Wu invariant is the first component of the bijection in Theorem \ref{Wu}.
For its precise definition, we outline the proof of the former assertion in Theorem \ref{Wu} as follows. 
For an immersion $f \colon M^3 \imm \R^5$, its normal bundle is obtained as the pull-back of the natural $\SO(2)$-bundle
\[\rho \colon \SO(5) \to \SO(5)/\SO(2)\]
by $\phi_{\tau, f} \colon M^3 \to \SO(5)/\SO(2)$.
Applying the Gysin exact sequence, we see that the Euler class of $\rho$ is of the form $2\Sigma$, where $\Sigma$ is a generator of $H^2(V_{5,3}; \Z) \cong \Z$.
Then, by the naturality of Euler classes, the normal Euler class $\chi$ of $f$ coincides with $2 (\phi_{\tau, f})^*(\Sigma)$.
Hence we have that $(\phi_{\tau, f})^*(\Sigma) \in \Gamma_2(\chi)$.

\begin{definition}[Wu invariant {\cite[Theorem 2]{Wu64}}, {\cite[Definition 4]{SST02}}]\label{Wuinv}
Let $f \colon M^3 \imm \R^5$ be an immersion with normal Euler class $\chi$, and $\phi_{\tau, f} \colon M^3 \to V_{5,3}$ the induced map.
Then the {\it Wu invariant} of $f$ with respect to $\tau$ is defined as the second cohomology class
\[c_\tau(f) \coloneqq (\phi_{\tau, f})^*(\Sigma) \in \Gamma_2(\chi),\]
where $\Sigma$ is the generator of $H^2(V_{5,3}; \Z) \cong \Z$ so that the Euler class of the natural $\SO(2)$-bundle $\rho \colon \SO(5) \to \SO(5)/\SO(2)$ coincides with $2\Sigma$.
\end{definition}

Clearly, the class $c_\tau(f)$ is invariant up to regular homotopy.

\begin{remark}
In \cite{SST02}, the class $c_\tau(f)$ is simply denoted by $c(f)$ under the contract that a parallelization is fixed.
It is important to specify which is the fixed parallelization, since the class $c_\tau(f)$ depends on the choice of $\tau$. (see, e.g., below Definition 4 in \cite[p.18]{SST02}).
In our results, we will choice a certain parallelization introduced in \S \ref{section:al_ctct} (see Definition \ref{dfn:al-ctct-paral}), and discuss the way of switching of the Wu invariant by changing parallelizations (see Appendix \ref{app:switch}).
\end{remark}

\begin{remark}
For the case where $\chi = 0$, Saeki--Sz\u cs--Takase gave another expression of the Wu invariant, which is different from ours. See \cite[\S 3]{SST02} for details.
\end{remark}

\subsubsection{Smale-type invariant}\label{subsection:Smaletype}

The second component of the bijection in Theorem \ref{Wu} can be identified with the Smale invariant if the 3-manifold $M^3$ is the 3-sphere, but that was given in a complicated way originally.
Saeki--Sz\u cs--Takase \cite{SST02} gave its geometrically clear expression for the case where $\chi = 0$ (i.e., the normal bundle of a given immersion is trivial).
After that, Juh\'asz generalized it for arbitrary $\chi$ \cite{Juh05}.
In below, we review only a part of \cite[\S 5]{SST02}.

Hereafter we assume that {\it the $3$-manifold $M^3$ is connected and closed}. 
For an immersion $f \colon M^3 \imm \R^5$, a {\it singular Seifert surface} of $f$ is a generic map $\hat{f} \colon X^4 \to \R^6_+$ from a compact oriented 4-manifold $X^4$ bounded by $M^3$ satisfying that $\hat{f}^{-1}(\R^5) = M^3$, $\hat{f}|_{M^3} = f$, and $\hat{f}$ has no singularities near the boundary.

\begin{definition}[{\cite[Definition 5]{SST02}}]\label{dfn:alpha}
We define an integer $\alpha(M^3)$ to be the dimension of the $\Z_2$-vector space $T \otimes_\Z \Z_2$, where $T$ is the torsion subgroup of $H_1(M^3; \Z)$.
\end{definition}

\begin{definition}[Smale-type invariant {\cite[Definitions 5 and 7]{SST02}}]\label{Smaletype}
Let $f \colon M^3 \imm \R^5$ be an immersion with trivial normal bundle and $\hat{f} \colon X^4 \to \R^6_+$ a singular Seifert surface of $f$.
Then define the {\it Smale-type invariant} of $f$ to be
\[i(f) \coloneqq \frac{3}{2}(\sigma(X^4) - \alpha(M^3)) + \frac{1}{2}(3t(\hat{f}) - 3l(\hat{f}) + L_\nu(f)) \in \Z,\]
where $t(\hat{f})$, $l(\hat{f})$, and $L_\nu(f)$ are integers obtained from the singularity of $\hat{f}$ and the double locus of $f$ (see \cite[\S 2 and 5]{SST02}).
\end{definition}

We here note only that if $\hat{f}$ can be taken as an embedding, the invariants $t(\hat{f})$, $l(\hat{f})$, and $L(f)$ all vanish.
Saeki--Sz\u cs--Takase showed that the integer $i(f)$ is well-defined and invariant up to regular homotopy. They also showed the following.

\begin{theorem}[{\cite[Theorem 6]{SST02}}]\label{complete}
{\it
Let $\tau$ be a parallelization of $M^3$.
Then the correspondence
\[(c_\tau, i) \colon \Imm[M^3, \R^5]_0 \to \Gamma_2(0) \times \Z\]
is a bijection, that is, the regular homotopy class of a given immersion with trivial normal bundle is determined by its Wu and Smale-type invariants.
}
\end{theorem}

\section{Contact and almost contact structures}\label{section:al_ctct}

In this section, we recall the notions of contact and almost contact structures.
See \cite{EM02, Gei08} for details.
We also make a preparation for our main results.

\subsection{Contact structures}

\begin{definition}[contact manifold]
Let $M^{2m+1}$ be a $(2m+1)$-manifold.
A {\it contact structure} on $M^{2m+1}$ is a maximally non-integrable hyperplane distribution $\xi \subset TM^{2m+1}$, i.e., if $\xi$ is locally defined by a 1-form $\alpha$ as $\xi = \Ker \alpha$, then the $(2m+1)$-form $\alpha \wedge (d\alpha)^m$ nowhere vanishes.
We call the pair $(M^{2m+1}, \xi)$ a {\it contact manifold}.
In the case where $\xi$ is cooriented (i.e., $\xi$ is globally defined by a 1-form), $(M^{2m+1}, \xi)$ is also said to be {\it cooriented}.
\end{definition}

\begin{remark}\label{rmk:cfml}
For a contact manifold $(M^{2m+1}, \xi)$, a defining 1-form $\alpha$ of $\xi$ induces the symplectic structure $d\alpha|_\xi$ on $\xi$.
Notice that the conformal class of $d\alpha|_\xi$ depends only on $\xi$.
Indeed, for any non-vanishing function $f$ on $M^{2m+1}$, it holds that $d(f\alpha)|_\xi = f \cdot d\alpha|_\xi$.
\end{remark}

Let $J_0$ be the standard complex structure on $\C^3$.

\begin{example}
Let $S^5_\epsilon \subset \C^3$ be the 5-sphere with radius $\epsilon > 0$ and centered at the origin.
Then define a hyperplane distribution $\xi_\mathrm{std} \subset TS^5_\epsilon$ as the complex tangency 
\[\xi_\mathrm{std} = TS^5_\epsilon \cap J_0(TS^5_\epsilon).\]
This forms a cooriented contact structure on standardly oriented $S^5_\epsilon$. We call this the {\it standard contact structure} on the 5-sphere.
\end{example}

\begin{example}\label{eg:link-ctct-strc}
Let $K$ be the link of an isolated surface singularity in $\C^3$.
Then define a hyperplane distribution $\xi_\mathrm{can} \subset TK$ as the complex tangency
\[\xi_\mathrm{can} = TK \cap J_0(TK).\]
This forms a cooriented contact structure on standardly oriented $K$ as well. We call this the {\it canonical contact structure} on the link $K$.
\end{example}

We can verify that the inclusion map $f \colon K \into S^5_\epsilon$ forms a {\it contact embedding}, i.e., 
\[TK \cap \xi_\mathrm{std}|_K = \xi_\mathrm{can}.\]

\subsection{Almost contact structures}

The following notion, which is weaker than the notion of contact structures, will fit to our later arguments.

\begin{definition}[almost contact manifold, e.g., {\cite[p.100, \S 10.1.B]{EM02}}]\label{def:al_ctct}
Let $M^{2m+1}$ be a $(2m+1)$-manifold.
An {\it almost contact structure} on $M^{2m+1}$ is the pair $(\xi, \omega)$ of a hyperplane distribution $\xi \subset TM^{2m+1}$ and a symplectic form $\omega$ on $\xi$ valued in the line bundle $TM^{2m+1}/\xi$.
We call the triplet $(M^{2m+1}; \xi, \omega)$ an {\it almost contact manifold}.
In the case where $\xi$ is cooriented, $(M^{2m+1}; \xi, \omega)$ is also said to be {\it cooriented}.
\end{definition}

\begin{example}
Let $(M^{2m+1}, \xi)$ be a contact manifold. On each point $x \in M^{2m+1}$, we have the hyperplane field $\xi$ and the symplectic structure $d\alpha|_\xi$ for some defining 1-form $\alpha$ of $\xi$ (Remark \ref{rmk:cfml}).
Hence every contact manifold determines an almost contact structure up to conformal equivalence.
\end{example}

For an oriented real vector space $V$ of dimension $2m+1$, let us call the pair $(H, \Omega)$ of a cooriented hyperplane $H$ and its symplectic structure $\Omega$ as a {\it symplectic hyperplane}.
A symplectic hyperplane $(H, \Omega)$ is said to be {\it positive} or {\it negative} if the orientation $\Omega^m$ on $H$ and the coorientation on $H$ form the orientation on $V$ or its reversion, respectively. 
Let $\mathcal{C}_+(V)$ denote the space of all positive symplectic hyperplanes in $V$.

\begin{proposition}\label{moduli}
{\it 
For any oriented real vector space $V$ of dimension $2m+1$, the space $\mathcal{C}_+(V)$ is homotopy equivalent to $\SO(2m+1)/\U(m)$.
Therefore, every cooriented and positive almost contact structure on $M^{2m+1}$ defines a section of the $\SO(2m+1)/\U(m)$-bundle over $M^{2m+1}$ associated to $TM^{2m+1}$.
}
\end{proposition}


\noindent{Here the unitary subgroup $\U(m)$ is embedded into the special orthogonal group $\SO(2m)$ as the subgroup $\{A \in \SO(2m) \mid J_0A = AJ_0\}$, where $J_0$ is the standard complex structure on $\R^{2m} = \C^m$.}
Also $\SO(2m)$ is embedded into $\SO(2m+1)$ in the standard way.

\begin{remark}[low-dimensional cases]\label{3-dim}
If $m = 1$, then we have that $\SO(3)/\U(1) \cong S^2$. 
This means, roughly saying, that a symplectic structure on a 2-plane determines only an orientation of the plane up to homotopy.
For the case where $m = 2$, we describe the identification $\mathcal{C}_+(V) \simeq \SO(5)/\U(2)$ in \S \ref{section:Wuinv}.
\end{remark}

We also generalize the notion of contact embedding as follows.

\begin{definition}[almost contact embedding]
Let $(M^{2m+1}; \xi, \omega)$ and $(N^{2n+1}; \eta, \varpi)$ be almost contact manifolds. 
An {\it almost contact embedding} of $(M^{2m+1}; \xi, \omega)$ into $(N^{2n+1}; \eta, \varpi)$ is an embedding $f \colon M^{2m+1} \into N^{2n+1}$ satisfying that 
\[Tf(M^{2m+1}) \cap \eta|_{f(M^{2m+1})} = df(\xi)\]
and $(df(\xi), \omega)$ is a symplectic subspace field of $(\eta|_{f(M^{2m+1})}, \varpi')$, where $\varpi'$ is a symplectic structure belonging to the same conformal class as $\varpi$.
\end{definition}

The following notion will help us to investigate properties of Wu invariants of almost contact embeddings (\S \ref{subsection:Wuinv-ctct-emb}) and also other immersions (\S \ref{subsection:Wuinv-cpx-tan}).

\begin{definition}[almost contact parallelization]\label{dfn:al-ctct-paral}
Let $(M^3; \xi, \omega)$ be a cooriented almost contact manifold such that $\xi$ is trivializable as a 2-plane field.
A parallelization $\tau \colon TM^3 \xto{\cong} M^3 \times \R^3$ of $M^3$ is said to be an {\it almost contact parallelization} if the following holds via $\tau$:
\begin{itemize}
\item[(i)] the constant vector field $e_1$ gives the coorientation of $\xi$;
\item[(ii)] the constant vector fields $e_2$ and $e_3$ span $\xi$, and $(e_2, e_3)$ gives the orientation of $\xi$ determined by $\omega$.
\end{itemize}
\end{definition}

\subsection{Reconstruction of Kinjo's immersions}\label{subsection:reconstruction-identification}

At the end of this section, we reconstruct the immersion $g_G \colon M(G) \imm \R^4$ introduced in \S \ref{subsubsection:imm-Kinjo}.
We also identify $M(G)$ with the link $K$ of the corresponding singularity.
This construction motivates the study of the Wu invariant in \S \ref{subsection:Wuinv-cpx-tan}.

\subsubsection{Reconstruction for the $A_1$-case}

We consider the {\it Whitney 2-sphere} \cite{Whi44}
\[w \colon S^2 \imm \C^2, \quad w(x_1, x_2, y) = (1 + y\sqrt{-1}) (x_1, x_2).\]
This map has only one positively transverse double point $h_0(0, 0, \pm 1) = (0, 0)$.
Moreover, this is a real analytic and totally real immersion with respect to the standard complex structure $J_0$ on $\C^2$.
We take the disk normal bundle of $w$.
Since the normal Euler number of $w$ is $-2$, the total space is nothing but the disk cotangent bundle $DT^*S^2 = X(A_1)$.
Then we have immersions
\[\hat{g}_{A_1} \colon X(A_1) \imm \C^2 \quad \text{and} \quad g_{A_1} = \hat{g}_{A_1}|_{M(A_1)} \colon M(A_1) \imm \C^2.\]
We now induce a complex structure $J$ on $X(A_1)$ by pulling-back $J_0$, which is the complexification of the zero section $S^2$. We also define an almost contact structure $\xi'_\mathrm{can}$ on $M(A_1)$ by
\[\xi'_\mathrm{can} = TM(A_1) \cap J(TM(A_1))\]
(we endow $\xi'_\mathrm{can}$ with the orientation as a complex line field).
Note that $\hat{g}_{A_1}$ is holomorphic, and hence it holds that $dg_{A_1}(\xi'_\mathrm{can}) \subset T\C^2$ is complex everywhere. 
We will focus on this property (see \S \ref{subsection:Wuinv-cpx-tan}).


\subsubsection{Reconstruction for the general case}

Let $G$ be one of Dynkin diagrams of types $A$-$D$-$E$.
We consider the union of 2-spheres embedded in some Euclidean space according to the following:
\begin{itemize}
\item corresponding to each vertex of $G$, embed one 2-sphere;
\item corresponding to each edge of $G$, arrange two 2-spheres so that they transversely intersect at one point.
\end{itemize}
Then we construct a map
\[w_G \colon S^2 \cup \dots \cup S^2 \to \C^2\]
such that each restriction $w_G|_{S^2}$ coincides with the Whitney 2-sphere $w$ after some rotation and parallel shift and is still totally real, and all 2-spheres transversely intersect in $\C^2$.
Then taking a regular neighborhood of $w_G$, we have immersions 
\[\hat{g}_G \colon X(G) \imm \C^2 \quad \text{and} \quad g_G = \hat{g}_G|_{M(G)} \colon M(G) \imm \C^2.\]
We now induce a complex structure $J$ on $X(G)$ by pulling-back the standard complex structure $J_0$ on $\C^2$.
We also define an almost contact structure $\xi'_\mathrm{can}$ on $M(G)$ by
\[\xi'_\mathrm{can} = TM(G) \cap J(TM(G))\]
(we endow $\xi'_\mathrm{can}$ with the orientation as a complex line field).
Then the immersion $\hat{g}_G$ is holomorphic, and hence it holds that $dg_G(\xi'_\mathrm{can}) \subset T\C^2$ is complex everywhere.

\subsubsection{Identification}\label{subsubsection:identification}

We claim that {\it there exists a diffeomorphism between $M(G)$ and $K$ such that $\xi'_\mathrm{can}$ and $\xi_\mathrm{can}$ are homotopic as almost contact structures via the diffeomorphism}. 
Let us show it as follows.
Consider the complex structure on the Milnor fiber $F$ as a complex submanifold of $\C^3$. Then we can find a spine of $F$, which is the transverse union $S^2 \cup \dots \cup S^2$ associated with $G$ as above, and all 2-spheres are totally real in $F$.
We take an embedding
\[\iota \colon X(G) \into F\]
such that $\iota$ identifies the spines of $X(G)$ and $F$, and $F$ deformation retracts to $\iota(X(G))$ by an isotopy.
Since both complex structures on $X(G)$ and $F$ have their spines as totally real, they are homotopic as almost complex structures.
Moreover, along the isotopy, almost complex structures and hence almost contact structures are deformed.
Consequently, we obtain an isotopy between the boundaries $(\iota(M(G)), d\iota(\xi'_\mathrm{can}))$ and $(K, \xi_\mathrm{can})$ such that $d\iota(\xi'_\mathrm{can})$ and $\xi_\mathrm{can}$ are homotopic as almost contact structures via the isotopy.
This concludes the claim, and we identify them as almost contact manifolds.

\section{General properties of the Wu invariant}\label{section:Wuinv}

As the preparation to show our main results, we state geometric properties of the Wu invariant.
Throughout this section, let $(M^3; \xi, \omega)$ denote an arbitrary closed cooriented almost contact 3-manifold.
We also introduce the following conventions.
Let $e_i$ denote the $i$-th standard vector of $\R^5$.

\begin{convention}\label{convSO2}
We embed $\SO(2)$ into $\SO(5)$ in the lowest diagonal position, i.e.,
\[\{E_3\} \times \SO(2) \subset \SO(5).\]
We also identify $\SO(5)/\SO(2)$ with $V_{5,3}$ via the induced map from the action of $\SO(5)$ on $V_{5,3}$ at $\begin{bmatrix}
e_1 & e_2 & e_3 
\end{bmatrix}$:
\[\SO(5) \to V_{5,3}, \quad 
\begin{bmatrix}
v_1 & v_2 & v_3 & v_4 & v_5 
\end{bmatrix}
\mapsto
\begin{bmatrix}
v_1 & v_2 & v_3
\end{bmatrix}.\]
\end{convention}

In addition, let $(\epsilon_1, \dots, \epsilon_5)$ be the basis of $(\R^5)^*$ dual to $(e_1, \dots, e_5)$, and let $(\hy)^* \colon \R^5 \to (\R^5)^*$ denote the linear isomorphism sending $e_i$ to $\epsilon_i$.

\begin{convention}\label{convU2}
We embed $\U(2)$ into $\SO(5)$ in the lower diagonal position, i.e.,
\[\{1\} \times \U(2) \subset \SO(5).\]
Moreover, considering the standard orientation on $\R^5$, we identify $\SO(5)/\U(2)$ with $\mathcal{C}_+(\R^5)$ (up to homotopy equivalence) via the induced map from the action of $\SO(5)$ on $\mathcal{C}_+(\R^5)$ at $(\Ker \epsilon_1, \epsilon_2 \wedge \epsilon_3 + \epsilon_4 \wedge \epsilon_5)$:
\[\SO(5) \to \mathcal{C}_+(\R^5), \quad 
\begin{bmatrix}
v_1 & v_2 & v_3 & v_4 & v_5 
\end{bmatrix}
\mapsto
(H, \Omega) = (\Ker (v_1^*), v_2^* \wedge v_3^* + v_4^* \wedge v_5^*),\]
where $H$ is cooriented in the $v_1$-direction.
\end{convention}

\subsection{Wu invariant and almost contact embeddings}\label{subsection:Wuinv-ctct-emb}

In this subsection, we show the following. The motivative example is that in \S \ref{subsubsection:imm-incl}.

\begin{proposition}\label{Wuinv_ctct_emb}
{\em
Let $(\R^5; \eta, \varpi)$ be 5-space endowed with an almost contact structure and 
\[f \colon (M^3; \xi, \omega) \into (\R^5; \eta, \varpi)\]
an almost contact embedding.
Then the Wu invariant $c_\tau(f) \in H^2(M^3; \Z)$ vanishes for any almost contact parallelization $\tau$ of $(M^3; \xi, \omega)$.
}
\end{proposition}

Before the proof, here are three remarks.

{\bf (a)} The triviality of $\xi$ as a 2-plane field is ensured as follows. 
In general, for an almost contact manifold $(M^{2n-1}; \xi, \omega)$, the hyperplane field $\xi$ admits a complex structure compatible with the symplectic structure $\omega$, which is unique up to homotopy.
Hence its first Chern class $c_1(\xi)$ makes sense.
Then the following is known:

\begin{theorem}[{\cite[Theorem 1.3]{Kas16}}]\label{Kas1.3}
{\it 
If a closed contact manifold $(M^{2n-1}, \xi)$ is a contact submanifold of a cooriented contact manifold $(N^{2n+1}, \eta)$ such that $H^2(N^{2n+1}; \Z) = 0$, then $c_1(\xi) = 0$.
}
\end{theorem}

\noindent{Notice that Theorem \ref{Kas1.3} holds also for almost contact manifolds.}
Applying this result to the case where $n = 2$, we have the triviality of our $\xi$ as a 2-plane field.

{\bf (b)} Since the map $f \colon M^3 \into \R^5$ is an embedding, the normal Euler class of $f$ vanishes \cite[Theorem 11.3]{MS74}. 
Hence the Wu invariant of $f$ is a 2-torsion for any parallelization.
However, it does not imply the Wu invariant itself vanishes.

{\bf (c)} Theorem \ref{Wuinv_ctct_emb} is related to the following result:

\begin{theorem}[{\cite[Theorem 1.5]{Kas16}}]\label{Kas1.5}
{\it
Let $(M^3, \xi)$ be a closed cooriented contact 3-manifold with $c_1(\xi) = 0$. 
Then there is a contact structure $\eta$ on $\R^5$ such that we can embed $(M^3, \xi)$ in $(\R^5, \eta)$ as a contact submanifold.
}
\end{theorem}

\noindent{In the proof of Theorem \ref{Kas1.5}, such an embedding was constructed so that its Wu invariant vanishes for an almost contact parallelization.}\\

The key to the proof of Theorem \ref{Wuinv_ctct_emb} is to consider the quotient map
\[\pi \colon V_{5,3} \cong \SO(5)/\SO(2) \to \SO(5)/\U(2),\]
and the composition
\[\psi_{\tau, f} \coloneqq \pi \circ \phi_{\tau, f} \colon M^3 \to \SO(5)/\U(2).\]
We obtain a new expression of the Wu invariant as follows.
Recall that $\Sigma \in H^2(V_{5,3}; \Z)$ is the generator chosen in Definition \ref{Wuinv}.

\begin{lemma}\label{Wuinv_expA}
{\em
Let $f \colon M^3 \imm \R^5$ be an immersion and $\tau$ a parallelization of $M^3$.
Then the Wu invariant $c_\tau(f)$ is equal to the class
\[(\psi_{\tau, f})^*(\Sigma') \in H^2(M^3; \Z),\]
where $\Sigma'$ is the generator of $H^2(\SO(5)/\U(2); \Z) \cong \Z$ chosen so that $\pi^*(\Sigma') = \Sigma$.
}
\end{lemma}

\begin{proof}
It suffices to show that the induced map
\[\pi^* \colon H^2(\SO(5)/\U(2); \Z) \to H^2(\SO(5)/\SO(2); \Z)\]
is an isomorphism.
This can be shown by applying the Gysin exact sequence to the orientable $S^3$-bundle $\pi \colon \SO(5)/\SO(2) \to \SO(5)/\U(2)$.
\end{proof}

\begin{remark}
There is a diffeomorphism $\SO(5)/\U(2) \cong \C P^3$ \cite[\S 8.1]{Gei08}.
Under this identification, the generator $\Sigma'$ coincides with the Poincar\'e dual of $\C P^2$ up to sign.
\end{remark}

\begin{proof}[Proof of Proposition \ref{Wuinv_ctct_emb}]
Let us consider the map 
\[\phi_{\tau, f} = \begin{bmatrix} v_1 & v_2 & v_3 \end{bmatrix} \colon M^3 \to V_{5,3} \cong \SO(5)/\SO(2)\]
induced from the differential of $f$ and the parallelization $\tau$, where $v_i(x) = df_x(e_i)$ for $i = 1, 2, 3$.
Here notice that $df(\xi)$ is spanned by $v_2$ and $v_3$, and has the orientation $v_2^* \wedge v_3^*$.
We also choose the coorientation of $\eta$ so that it is compatible to that of $\xi$.
Then the vector field $v_1$ gives the coorientation of $\eta$ on $f(M^3)$.
Moreover, according to Proposition \ref{moduli}, the almost contact structure $(\eta, \varpi)$ is regarded as a smooth map
\[(\eta, \varpi) \colon \R^5 \to \mathcal{C}_+(\R^5).\]

We now consider the quotient map $\pi \colon \SO(5)/\SO(2) \to \SO(5)/\U(2)$ and the composition
\[\psi_{\tau, f} \coloneqq \pi \circ \phi_{\tau, f} \colon M^3 \to \SO(5)/\U(2).\]
By Lemma \ref{Wuinv_expA}, it suffices to show that the map $\psi_{\tau, f}$ is null-homotopic. More precisely, the lower left triangle in the following diagram commutes, where the maps $^*1$ and $^*2$ are the identification fixed in Conventions \ref{convSO2} and \ref{convU2}.
\begin{center}
\begin{tikzcd}
M^3 \arrow[d, "f"] \arrow[rrd, "{\psi_{\tau, f}}" description] \arrow[r, "{\phi_{\tau, f}}"] & {V_{5,3}} \arrow[r, "^*1"]             & \SO(5)/\SO(2) \arrow[d, "\pi"] \\
\R^5 \arrow[r, "{(\eta, \varpi)}"']                                               & \mathcal{C}_+(\R^5) \arrow[r, "^*2"'] & \SO(5)/\U(2)                  
\end{tikzcd}
\end{center}
Since the map $f$ is an almost contact embedding, the distribution $df(\xi) \subset (\eta, \varpi)|_{f(M^3)}$ forms a symplectic subspace field, and the 2-frame $(v_2, v_3)$ is homotopic to a symplectic basis field of $(df(\xi), \varpi|_{df(\xi)})$ (recall Remark \ref{3-dim}). 
Then for each $x \in M^3$, by the extension of symplectic basis, the symplectic structure $\varpi_{f(x)}$ on $\eta_{f(x)}$ has the form
\[v_2(x)^* \wedge v_3(x)^* + v_4(x)^* \wedge v_5(x)^*\]
up to homotopy, where $v_4(x)$ and $v_5(x)$ are appropriate vectors.
Then we have that the restriction $(\eta, \varpi)|_{f(M^3)}$ has the representative
\[\begin{bmatrix} v_1 & v_2 & v_3 & v_4 & v_5 \end{bmatrix}\]
in the sense of Convention \ref{convU2}, up to homotopy.
By the form of $\phi_{\tau, f}$ and Convention \ref{convSO2}, we find the commutativity of the lower left triangle.
This completes the proof.
\end{proof}

\subsection{Wu invariant and complex tangency}\label{subsection:Wuinv-cpx-tan}

In this subsection, we show the following. The motivative example is introduced in \S \ref{subsubsection:imm-Kinjo} and \S \ref{subsection:reconstruction-identification}.
Let $J_0$ be the standard complex structure on $\R^4 = \C^2$:
\[
J_0 \coloneqq 
\begin{bmatrix}
0 & -1 & 0 & 0 \\
1 & 0 & 0 & 0 \\
0 & 0 & 0 & -1 \\
0 & 0 & 1 & 0
\end{bmatrix}.
\]
Let $j \colon \R^4 \into \R^5, \ (x_1, x_2, x_3, x_4) \mapsto (x_1, x_2, x_3, x_4, 0)$ be the standard inclusion.

\begin{proposition}\label{Wuinv_hol_imm}
{\em 
Assume that $\xi$ is trivial as a 2-plane field.
Let 
\[g \colon (M^3; \xi, \omega) \imm (\C^2, J_0)\]
be an immersion satisfying that $dg(\xi) \subset T\C^2$ forms a complex line everywhere.
Then the Wu invariant $c_\tau(j \circ g) \in H^2(M^3; \Z)$ vanishes for any almost contact parallelization $\tau$ of $(M^3; \xi, \omega)$.
}
\end{proposition}

To prove this proposition, we consider the map
\[\phi_{\tau, j \circ g} \colon M^3 \to \SO(5)/\SO(2),\]
and an additional map
\[\Phi_{\tau, g} \colon M^3 \to V_{4,3} \cong \SO(4),\] 
which is induced from the differential of $g$ and the parallelization $\tau$.
Since the 5th column of the map $\phi_{\tau, j \circ g}$ is constantly $e_5$, 
the relationship between these maps is as follows.
\begin{center}
\begin{tikzcd}
\SO(4) \arrow[r, "\iota", hook]                                          & \SO(5) \arrow[d, "\rho"] \\
M^3 \arrow[r, "{\phi_{\tau, j \circ g}}"'] \arrow[u, "{\Phi_{\tau, g}}"] & \SO(5)/\SO(2)           
\end{tikzcd}
\end{center}
Here $\iota \colon \SO(4) \into \SO(5)$ is the embedding in the {\it upper} diagonal position and $\rho \colon \SO(5) \to \SO(5)/\SO(2)$ is the quotient map (recall Convention \ref{convSO2}).
Then we obtain another expression of the Wu invariant as follows.
Consider the usual embedding of $\U(2)$ into $\SO(4)$ defined by $\U(2) = \{A \in \SO(4) \mid J_0 A = A J_0\}$.

\begin{lemma}\label{Wuinv_expB}
{\em
Let $g \colon M^3 \imm \C^2$ be an immersion and $\tau$ a parallelization of $M^3$.
Then the Wu invariant $c_\tau(j \circ g)$ is equal to
\[(\Phi_{\tau, g})^* [\U(2)]^* \in H^2(M^3; \Z),\]
where $[\U(2)]^*$ is the Poincar\'e dual of $\U(2) \subset \SO(4)$.
}
\end{lemma}

\begin{proof}
Notice that the induced map
\[(\rho \circ \iota)^* \colon H^2(\SO(5)/\SO(2); \Z) \xto{\rho^*} H^2(\SO(5); \Z) \xto{\iota^*} H^2(\SO(4); \Z),\]
which is regarded as $\Z \to \Z_2 \to \Z_2$, is a surjection.
Indeed, applying the Gysin exact sequence of the orientable $\SO(2)$-bundle $\rho$, we see that $\rho^*$ is a surjection; since the 3-skeletons of $\SO(4)$ and $\SO(5)$ are the same via the inclusion $\iota$, we see that $\iota^*$ is an isomorphism.

Next, we verify that the Poincar\'e dual $[\U(2)_J]^*$ generates $H^2(\SO(4); \Z) \cong \Z_2$ as follows.
Applying the Gysin exact sequence to the orientable $S^3$-bundle
\[\pi' \colon \SO(4) \to \SO(4)/\SU(2) \cong \R P^3,\]
we have that 
$(\pi')^* \colon H^2(\SO(4)/\SU(2); \Z) \to H^2(\SO(4); \Z)$
is an isomorphism.
Moreover, applying the Gysin exact sequence to the orientable $S^1$-bundle
\[\pi'' \colon \SO(4)/\SU(2) \to \SO(4)/\U(2) \cong S^2,\]
we have that 
$(\pi'')^* \colon H^2(S^2; \Z) \to H^2(\SO(4)/\SU(2); \Z)$
is a surjection.
Hence the generator of $H^2(\SO(4); \Z)$ is obtained by
\[(\pi')^* \circ (\pi'')^* [\pt]^* = [(\pi'' \circ \pi')^{-1}(\pt)]^* = [\U(2)]^*.\]
Thus we see that
\[c_\tau(j \circ g) = (\phi_f)^*(\Sigma) = (\Phi_{\tau, g})^* (\rho \circ \iota)^*(\Sigma) = (\Phi_{\tau, g})^* [\U(2)]^*
,\]
which completes the proof.
\end{proof}

\begin{proof}[Proof of Proposition \ref{Wuinv_hol_imm}]
Let
\[\Phi_{\tau, g} = \begin{bmatrix} v_1 & v_2 & v_3 & v_4 \end{bmatrix} \colon M^3 \to V_{4,3} \cong \SO(4)\] 
denote the map induced from the differential of $g$ and the parallelization $\tau$, where $v_i(x) = dg_x(e_i)$ for $i = 1, 2, 3$, and $v_4$ is the normal vector field to $g$.

By Lemma \ref{Wuinv_expB}, it suffices to show that $\Phi_{\tau, g}$ does not meet with $\U(2)$.
Let $x \in M^3$. By the form of $\tau$, the 2-plane spanned by $v_2(x)$ and $v_3(x)$ coincides with 
\[dg_x(\xi_x) \subset T_{g(x)}\C^2 = \C^2.\]
By assumption, this 2-plane is complex with respect to $J_0$.
Hence we observe that the 2-plane spanned by $v_3(x)$ and $v_4(x)$ is not complex with respect to $J_0$.
Now we assume that the point $x \in M^3$ satisfies that $\Phi_{\tau, g}(x) \in \U(2)$, i.e., 
\[J_0 \cdot \Phi_{\tau, g}(x) = \Phi_{\tau, g}(x) \cdot J_0.\]
Then we have that $v_4(x) = J_0 \cdot v_3(x)$.
This makes the contradiction to the earlier observation, which completes the proof. 
\end{proof}

\section{Main results}\label{section:main}

In this section, we prove our main results and give their applications.

\subsection{On Main Theorem \ref{main:determination-incl}}\label{section:determination-incl}

We consider inclusion maps
\[\hat{f} \colon F \into \R^6_+ \quad \text{and} \quad f \coloneqq \hat{f}|_{K} \colon K \into \R^5\]
defined in \S \ref{subsubsection:imm-incl}.
We also notice that the normal bundle of $f$ is trivial (see (b) of remarks in below Proposition \ref{Wuinv_ctct_emb}).
Therefore the Smale-type invariant of $f$ makes sense.

\begin{proof}[Proof of Main Theorem \ref{main:determination-incl}]
Even if we remove one point of $S^5_\epsilon$, the map $f \colon (K, \xi_\mathrm{can}) \into (\R^5, \xi_\mathrm{std}|_{\R^5})$ is still a contact embedding.
Then the vanishing of the Wu invariant $c_\tau(f)$ follows from Theorem \ref{Wuinv_ctct_emb}.

On the other hand, the Smale-type invariant $i(f)$ is computed by choosing the inclusion map $\hat{f} \colon F \into \R^6_+$ as a singular Seifert surface of $f$.
\end{proof}

\begin{remark}\label{rmk:intrinsic}
The integer $\alpha(K)$ is an intrinsic data.
Furthermore, the canonical contact structure $\xi_\mathrm{can}$ on $K$ is characterized as the Milnor fillable structure, which is unique up to contactomorphism \cite{CNP06}.
Thus, the regular homotopy class of the inclusion map is determined only by the topology of the link as an abstract 3-manifold and the signature of the Milnor fiber.
\end{remark}

\begin{remark}\label{rmk:MFB}
Main Theorem \ref{main:determination-incl} holds also for the boundary of the Milnor fiber of every non-isolated surface singularity in $\C^3$.
\end{remark}

We list the Smale-type invariants $i(f)$ for all types of simple singularities in Table \ref{data}.
Here, (co)homological data are computed from the plumbing manifold structures of $K \cong M(G)$ and $F \cong X(G)$.
We note that these data have been studied also in general (see, e.g., \cite{Nem99, GP24} for $\sigma(F)$ and \cite{ADHSZ03, NS12} for $\alpha(K)$).

\begin{table}[h]
\caption{Data of simple singularities and their immersions ($n \ge 2$)}
\label{data}
\begin{center}

\begin{tabular}{c||c|c|c|c}
type & $H^2(K; \Z)$ & $\sigma(F)$ & $\alpha(K)$ & $i(f)$ \\\hline
$A_{n-1}$ & $\Z_n$ & $-(n-1)$
&
\begin{tabular}{c}
1 ($n$ even) \\ 0 ($n$ odd)
\end{tabular}
&
\begin{tabular}{c}
$-3n/2 \ $ ($n$ even) \\ $-3(n-1)/2 \ $ ($n$ odd)
\end{tabular}
\\\hline
$D_{n+2}$
&
\begin{tabular}{c}
$\Z_2 \oplus \Z_2$ ($n$ even) \\ $\Z_4$ ($n$ odd)
\end{tabular} 
&
$-(n+2)$
&
\begin{tabular}{c}
2 ($n$ even) \\ 1 ($n$ odd)
\end{tabular}
&
\begin{tabular}{c}
$-3(n + 4)/2 \ $ ($n$ even) \\ $-3(n + 3)/2 \ $ ($n$ odd)
\end{tabular}
\\\hline
$E_6$ & $\Z_3$ & $-6$ & $0$ & $-9$ \\
$E_7$ & $\Z_2$ & $-7$ & 1 & $-12$ \\
$E_8$ & $0$ & $-8$ & 0 & $-12$ 
\end{tabular}
\end{center}
\end{table}

\subsection{On Main Theorem \ref{main:determination-Kinjo}}\label{section:determination-Kinjo}

We consider Kinjo's immersions
\[\hat{g}_G \colon X(G) \imm \C^2 \quad \text{and} \quad g_G \coloneqq \hat{g}_G|_{M(G)} \colon M(G) \imm \C^2\]
which are defined in \S \ref{subsubsection:imm-Kinjo} and reconstructed in \S \ref{subsection:reconstruction-identification}.
Notice that the normal bundle of $j \circ g_G$ is trivial.
Therefore the Smale-type invariant of $j \circ g_G$ makes sense.

\begin{proof}[Proof of Main Theorem \ref{main:determination-Kinjo}]
We consider the almost contact structure $\xi'_\mathrm{can}$ on $M(G)$, which was defined in \S \ref{subsection:reconstruction-identification}.
To show the vanishing of the Wu invariant, we will apply Proposition \ref{Wuinv_hol_imm} to $g_G$.
We verify that $\xi'_\mathrm{can}$ satisfies the assumptions of Proposition \ref{Wuinv_hol_imm} as follows.
Since we have the identification of $(M(G), \xi'_\mathrm{can})$ and $(K, \xi_\mathrm{can})$ as almost contact manifolds, it holds that $c_1(\xi'_\mathrm{can}) = c_1(\xi_\mathrm{can})$.
We also recall that $(K, \xi_\mathrm{can})$ admits a contact embedding into $(\R^5, \xi_\mathrm{std}|_{\R^5})$ by the inclusion map $f$. 
Then Theorem \ref{Kas1.3} ensures that $c_1(\xi_\mathrm{can}) = 0$. 
Hence $\xi'_\mathrm{can}$ is trivial as a 2-plane field.
Furthermore, $dg_G(\xi'_\mathrm{can}) \subset T\C^2$ was complex everywhere by construction.
Therefore we can apply Proposition \ref{Wuinv_hol_imm} to $g_G$.

To compute $i(j \circ g_G)$, we consider the push-forward of $j \circ \hat{g}_G \colon X(G) \imm \R^5$ into $\R^6_+$.
This is not an embedding, but by homotoping the spine of $X(G)$ in $\R^5$, this can be regularly homotoped to an embedding.
Moreover, pushing the image of interior of $X(G)$ into $\R^6_+$, the result map satisfies the condition to be a singular Seifert surface of $j \circ g_G$.
Then the assertion follows from $\sigma(X(G)) = -\# V(G)$.
\end{proof}

\subsection{Applications}\label{section:application}

We obtain the following as a corollary of Main Theorems \ref{main:determination-incl} and \ref{main:determination-Kinjo}, which refines and generalizes Corollary \ref{cor:reg-htpicity-AD}.

\begin{corollary}\label{cor:reg-htpicity}
{\em
For each type of simple singularities, its link $K$, and the corresponding Dynkin diagram $G$, identify $K$ and $M(G)$ as explained in \S \ref{subsubsection:identification}.
Then the inclusion map $f \colon K \into \R^5$ and the immersion $j \circ g_G \colon M(G) \imm \R^5$ have the same Wu and Smale-type invariants.
Therefore, these two immersions are regularly homotopic.
}
\end{corollary}

We also show the following by further argument, which recovers Theorem \ref{thm:Kin15}.
The following proof is due to G.~Pint\'er.

\begin{theorem}\label{thm:determination-Kinjo-Smale}
{\em
For each type of simple singularities and the corresponding Dynkin diagram $G$, the Smale invariant of $\bar{g}_G \circ p$ is
\[
\Omega(\bar{g}_G \circ p) = 
\begin{cases}
(n^2 - 1, 0) & \text{ for } A_{n-1} \ (n \ge 2); \\
(4n^2 + 12n - 1, 0) & \text{ for } D_{n+2} \ (n \ge 2); \\
(167, 0) & \text{ for } E_6; \\
(383, 0) & \text{ for } E_7; \\
(1079, 0) & \text{ for } E_8.
\end{cases} 
\]
}
\end{theorem}

\begin{proof}
Let us denote $\Omega(\bar{g}_G \circ p) = (a, b)$.
By Corollary \ref{cor:reg-htpicity} and Proposition \ref{thm:relationship-push},
\[\Omega(\Phi|_\mathfrak{S}) = \Omega(f \circ p) = \Omega(j \circ g_G \circ p) = -\Omega(j \circ \bar{g}_G \circ p) = -(a + 2b).\]
Since the left-hand side is computed by Theorem \ref{thm:NP15}, it suffices to compute only the integer $a$.
By Theorem \ref{thm:ET11}, it holds that $a = D(\bar{g}_G \circ p) - 1$, where $D(\bar{g}_G \circ p)$ is the normal mapping degree of $g_G \circ p$.
Since the covering degree of $p$ is equal to the order $\# \Gamma$ of the corresponding group $\Gamma$ (recall \S \ref{subsubsection:quotient}), we obtain that 
\[a = D(\bar{g}_G \circ p) - 1 = \# \Gamma \cdot D(\bar{g}_G) - 1.\] 
Moreover the immersion $g_G \colon M(G) \imm \R^4$ bounds the immersion $\hat{g}_G \colon X(G) \imm \R^4$. Hence
\[D(\bar{g}_G) = \chi(X(G)) = 1 + b_2(X(G)) = 1 + \# V(G)\]
(see, e.g., \cite[Theorem 2.2(b)]{KM99}), where $\# V(G)$ is the number of vertices of the corresponding Dynkin diagram $G$. 
Therefore we have that
\[a = \# \Gamma \cdot (1 + \# V(G)) - 1.\]
The proof is completed by computing the right-hand side for each case.
\end{proof}

\begin{remark}\label{rmk:argument}
In contrast to Kinjo's argument in \cite{Kin15}, we avoided to construct singular Seifert surfaces and count their singularities. The crucial point in our argument is to focus on almost contact structures on the 3-manifolds and their tangential properties in target spaces, and employ N\'emethi--Pint\'er's formula.
\end{remark}

We obtain two corollaries of Theorem \ref{thm:determination-Kinjo-Smale}.
The following corresponds to Corollary \ref{thm:Kin15-rev}.

\begin{corollary}\label{thm:determination-Kinjo-Smale-rev}
{\em
Under the same setup in Theorem \ref{thm:determination-Kinjo-Smale}, the Smale invariant of $g_G \circ p$ is
\[
\Omega(g_G \circ p) = 
\begin{cases}
(-n^2 - 1, 1) & \text{ for } A_{n-1} \ (n \ge 2); \\
(-4n^2 - 12n - 1, 1) & \text{ for } D_{n+2} \ (n \ge 2); \\
(-169, 1) & \text{ for } E_6; \\
(-385, 1) & \text{ for } E_7; \\
(-1081, 1) & \text{ for } E_8.
\end{cases}
\]
}
\end{corollary}

The following is immediate from the computation.
However, the author does not know any geometric reason for this phenomenon.

\begin{corollary}
{\em
Among immersions $g_G \circ p \colon S^3 \imm \R^4$ for all $A$-$D$-$E$ cases, only two immersions $g_{D_{17}} \circ p$ and $g_{E_8} \circ p$ are regularly homotopic.
}
\end{corollary}

Here notice that the corresponding assertion for immersions $\Phi|_\mathfrak{S} = f \circ p$, which are indeed regularly homotopic to $j \circ g_G \circ p$, is already known by Theorem \ref{thm:NP15}.

\appendix

\section{Switching of the Wu invariant}\label{app:switch}

In \S \ref{section:Wuinv}, we investigated the property of the Wu invariant with respect to an almost contact parallelization.
In this appendix, we quantify the switching of the Wu invariant when changing a given parallelization to another one. We also discuss the realization problem of a given cohomology class as the Wu invariant.
Let $M^3$ be an oriented 3-manifold and $f \colon M^3 \imm \R^5$ an immersion.
We show the following two.

\begin{proposition}\label{switch}
{\em
For two parallelizations $\tau_0, \tau_1 \colon TM^3 \xto{\cong} M^3 \times \R^3$ of $M^3$, it holds that
\[c_{\tau_1}(f) = c_{\tau_0}(f) + \beta(d(\tau_0, \tau_1)),\]
where $d(\tau_0, \tau_1) \in H^1(M^3; \pi_1(\SO(3))) \cong H^1(M^3; \Z_2)$ is the difference class between $\tau_0$ and $\tau_1$ (see \S \ref{subsection:setup-of-app}) 
and $\beta \colon H^1(M^3; \Z_2) \to H^2(M^3; \Z)$ is the Bockstein homomorphism.
}
\end{proposition}

\begin{proposition}\label{realize}
{\em
For the normal Euler class $\chi \in H^2(M^3; \Z)$ of $f$ and any class $c \in \Gamma_2(\chi)$, there exists a parallelization $\tau$ of $M^3$ such that $c_\tau(f) = c$.
}
\end{proposition}

\subsection{Setup}\label{subsection:setup-of-app}

We consider the homogeneous space $G_{5,3} \coloneqq \SO(5) / \{\SO(3) \times \SO(2)\}$,  the oriented Grassmannian. 
Then we have the natural quotient $V_{5,3} \to G_{5,3}$, the Gauss map $\gamma_f \colon M^3 \to G_{5,3}$ of the immersion $f \colon M^3 \imm \R^5$, and the following pull-back diagram.
\begin{center}
\begin{tikzcd}
(\gamma_f)^*\pi \arrow[d, dashed] \arrow[r, "\tilde{\gamma}_f", dashed] & {V_{5,3}} \arrow[d, "\pi"] \\
M^3 \arrow[r, "\gamma_f"]                                               & {G_{5,3}}                 
\end{tikzcd}
\end{center}
Notice that there is one-to-one correspondence
\[\{ \text{parallelizations of } M^3 \} \longleftrightarrow \{ \text{sections of } (\gamma_f)^*\pi \},\]
which is given by $\phi_{\tau, f} = \tilde{\gamma}_f \circ s$ for a parallelization $\tau$ and the corresponding section $s$.
Also notice that the bundle $(\gamma_f)^*\pi$ is trivializable, since this is a principal $\SO(3)$-bundle and has a section.

Let $d(\tau_0, \tau_1) \in H^1(M^3; \pi_1(\SO(3))) \cong H^1(M^3; \Z_2)$ denote the {\it difference class} between two parallelizations $\tau_0$ and $\tau_1$ of $M^3$ over the 1-skeleton of $M^3$, as sections of $(\gamma_f)^*\pi$.

\subsection{Proof of Proposition \ref{switch}}

We trivialize the bundle $(\gamma_f)^*\pi$ by $\phi_{\tau_0, f}$, i.e., 
so that $\tau_0$ corresponds to the constant section.
Then $\tilde{\gamma}_f \colon M^3 \times \SO(3) \to V_{5,3}$ is of the form
\[\tilde{\gamma}_f(x, A) = \phi_{\tau_0, f} \cdot A.\]
Furthermore, the other parallelization $\tau_1$ of $M^3$ forms the section
\[s_{\tau_1} \colon M^3 \to M^3 \times \SO(3), \quad s_{\tau_1}(x) = (x, g_{10}(x)),\]
where $g_{10} \colon M^3 \to \SO(3)$ is the transition function from $\tau_0$ to $\tau_1$ on the bundle $TM^3$.
Then the following diagram commutes.
\begin{center}
\begin{tikzcd}
M^3 \times \SO(3) \arrow[r, "\tilde{\gamma}_f"] & {V_{5,3}} \arrow[d, "\pi"] \\
M^3 \arrow[r, "\gamma_f"] \arrow[u, "s_{\tau_1}"] \arrow[ru, "{\phi_{\tau_1, f}}" description] & {G_{5,3}}                 
\end{tikzcd}
\end{center}
We can verify that $c_{\tau_1}(f) = (s_{\tau_1})^* \circ (\tilde{\gamma}_f)^* (\Sigma)$ from the diagram. Since it holds that
\[(\tilde{\gamma}_f)^*(\Sigma) = (c_{\tau_0}(f), \iota^*(\Sigma)) \quad \text{and} \quad (s_{\tau_1})^*(u, t) = u + (g_{10})^*(t),\]
where $\iota \colon \SO(3) \into V_{5,3}$ is the inclusion as the fiber of $V_{5,3} \to G_{5,3}$, we have that
\[c_{\tau_1}(f) = c_{\tau_0}(f) + (g_{10})^* \circ \iota^*(\Sigma).\]

To conclude Theorem \ref{switch}, we claim the following two.
Let $\kappa \in H^2(\SO(3); \Z) \cong \Z_2$ and $\lambda \in H^1(\SO(3); \Z_2) \cong \Z_2$ denote the generator, respectively.

\begin{claim}\label{isk}
{\em
The restriction $\iota^* \colon H^2(V_{5,3}, \Z) \to H^2(\SO(3); \Z)$ is a surjection, i.e., $\iota^*(\Sigma) = \kappa$.
}
\end{claim}

\begin{proof}
We decompose the inclusion $\iota \colon \SO(3) \into V_{5,3}$ into three maps as follows:
\[\iota = \rho \circ \iota' \colon \SO(3) \into \SO(5) \to V_{5,3},\]
where $\iota' \colon \SO(3) \into \SO(5)$ is the inclusion and $\rho \colon \SO(5) \to V_{5,3}$ is the natural quotient.
First, the restriction $(\iota')^* \colon H^2(\SO(5); \Z) \to H^2(\SO(3); \Z)$ is an isomorphism. Indeed, the 2-skeletons of $\SO(3)$ and $\SO(5)$ are the same. (In fact, both groups are isomorphic to $\Z_2$.)
Second, the pull-back $\rho^* \colon H^2(V_{5,3}; \Z) \to H^2(\SO(5); \Z)$ is a surjection. 
Indeed, applying the Gysin exact sequence to the orientable $\SO(2)$-bundle $\rho$, we have that
\[H^2(V_{5,3}; \Z) \xto{\rho^*} H^2(\SO(5); \Z) \to H^1(V_{5,3}; \Z) = 0.\]
Therefore $\iota^* = (\iota')^* \circ \rho^*$ is a surjection.
\end{proof}

\begin{claim}\label{td}
{\em
It holds that $(g_{10})^*(\lambda) = d(\tau_0, \tau_1) \in H^1(M^3; \Z_2)$.
}
\end{claim}

\begin{proof}
We regard $(g_{10})^*(\lambda) \in H^1(M^3; \Z_2) \cong \Hom(\pi_1(M^3), \Z_2)$.
Then this homomorphism detects a loop $\gamma \in \pi_1(M^3)$ whose image $(g_{10})_*(\gamma) \in \pi_1(\SO(3))$ is the generator $[\SO(2)] \in \pi_1(\SO(3)) \cong \Z_2$.
Hence $(g_{10})^*(\lambda)$ is the obstruction to homotope $g_{10}$ to the constant map on the 1-skeleton of $M^3$.
Therefore $(g_{10})^*(\lambda)$ coincides with the difference class $d(\tau_0, \tau_1)$.
\end{proof}

The equality of Claim \ref{td} is that of the upper row of the following diagram.
\begin{center}
\begin{tikzcd}
H^1(\SO(3); \Z_2) \cong \Z_2 \arrow[d, "\beta"] \arrow[r, "\tau^*"] & H^1(M^3; \Z_2) \arrow[d, "\beta"] \\
H^2(\SO(3); \Z) \cong \Z_2 \arrow[r, "\tau^*"]                      & H^2(M^3; \Z)                     
\end{tikzcd}
\end{center}
Taking down the equality by Bockstein homomorphisms, we have that
\[(g_{10})^*(\kappa) = \beta(d(\tau_0, \tau_1)) \in H^2(M^3; \Z).\]
Then Claims \ref{isk} and \ref{td} prove Theorem \ref{switch}.

\subsection{Proof of Proposition \ref{realize}}

This theorem is an immediate conclusion of the following and the surjectivity of the Bockstein homomorphism $\beta \colon H^1(M^3; \Z_2) \to \Gamma_2(0)$.

\begin{claim}
{\em
Fix a parallelization $\tau_0$ of $M^3$.
For any class $c' \in H^1(M^3; \Z_2)$, there exists a parallelization $\tau_1$ such that $c' = d(\tau_0, \tau_1)$.
}
\end{claim}

\begin{proof}
We regard $c' \in H^1(M^3; \Z_2) \cong \Hom(\pi_1(M^3), \Z_2)$.
We will find a map from $M^3$ to $\SO(3)$ which becomes a transition function from $\tau_0$ to a desired $\tau_1$.
We make a skeletonwise construction as follows.
On the 1-skeleton of $M^3$, we define a map $g_{01}' \colon \mathrm{sk}_1 M^3 \to \SO(3)$ so that for each simple loop $\gamma$ in the 1-skeleton of $M^3$, 
\begin{itemize} 
\item if $c'(\gamma) = 1$, then $g_{01}'|_\gamma \colon \gamma \to \SO(3)$ is a homeomorphism onto $\SO(2)$;
\item if $c'(\gamma) = 0$, then $g_{01}'|_\gamma \colon \gamma \to \SO(3)$ is null-homotopic.
\end{itemize}
We can extend $g_{01}'$ to a map $g_{01} \colon M^3 \to \SO(3)$, since each of homotopy groups
\[\pi_1(\mathrm{sk}_2 \SO(3), \mathrm{sk}_1 \SO(3)) = \pi_1(\R P^2, \SO(2)), \quad \pi_2(\SO(3), \mathrm{sk}_2 \SO(3)) = \pi_2(\SO(3), \R P^2)\]
vanishes.
We now define $\tau_1$ so that $g_{01}$ is the transition function from $\tau_0$ to $\tau_1$.
Then it holds that $c' = d(\tau_0, \tau_1)$ by construction.
\end{proof}

\end{document}